\documentclass[12pt]{article}
\usepackage{amsthm,amsfonts,amssymb,amsmath}
\usepackage{stmaryrd}
\usepackage[hidelinks]{hyperref}
\usepackage{epsfig}
\usepackage{url}
\usepackage{xcolor,tikz,dsfont}
\usetikzlibrary{positioning, arrows.meta,calc}
\usetikzlibrary{decorations.pathreplacing,decorations.markings}
\usetikzlibrary{matrix}
\usepackage{multicol,graphicx}

\usepackage{fullpage}
\usepackage{bbm}
\usepackage{pdflscape}
\usetikzlibrary{decorations}
\usepackage{svg}
\tikzset{
    vertex/.style={circle, draw, fill=black, inner sep=1.5pt},
    openvertex/.style={circle, draw, fill=white, inner sep=1.5pt},
    lbl/.style={draw=none, fill=none, circle=none, inner sep=1pt},
    arc/.style={-{Stealth[length=2.2mm,width=1.6mm]}, line width=0.8pt},
}

\tikzset{
    proofbox/.style={rounded corners=7pt, draw=black!45, fill=black!3, line width=0.5pt},
    addarc/.style={arc, red!70!black, line width=1.0pt},
    marked/.style={openvertex, line width=1.1pt},
    faint/.style={draw=black!40},
}

\usepackage[shortlabels]{enumitem}

\numberwithin{equation}{section}

\makeatletter
\let\c@figure\c@equation
\let\c@table\c@equation

\makeatother

\newtheorem{thm}{Theorem}[section]

\newtheorem{lem}[thm]{Lemma}

\newtheorem{conj}[thm]{Conjecture}

\theoremstyle{definition}

\theoremstyle{remark}

\newtheorem{rem}[thm]{Remark}

\DeclareTextCompositeCommand{\v}{OT1}{l}{l\nobreak\hspace{-.1em}'}
\DeclareTextCompositeCommand{\v}{OT1}{t}{t\nobreak\hspace{-.1em}'\nobreak\hspace{-.15em}}

\DeclareMathOperator{\Hom}{Hom}

\DeclareMathOperator{\rng}{rng}

\newcommand{\ent}{\mathrm{H}}
\newcommand{\Prob}{\mathbb{P}}

\setlist{leftmargin=3\parindent,labelindent=3\parindent}
\setlist[enumerate]{%
  leftmargin=3\parindent,%
  align=left,%
  labelwidth=3\parindent,%
  labelsep=0pt%
}
\setlist[enumerate,1]{%
  label={\normalfont (\arabic*)}, ref={\arabic*}%
}

\title{An entropy proof of the commonality of apices of trees}
\author{Hao Chen\thanks{School of Mathematical Sciences, Soochow University,
Suzhou 215006, China. E-mail:
\href{mailto:chenhao@suda.edu.cn}{\texttt{chenhao@suda.edu.cn}}}}
\date{}

\begin{document}

\maketitle

\begin{abstract}
A graph \(H\) is called \emph{common} if the number of monochromatic copies of \(H\) in a \(2\)-edge-colouring of a large complete graph is asymptotically minimized by the random colouring with an equal proportion of each colour.
Grzesik, Lee, Lidick\'y and Volec conjectured that every $k$-apex of any connected Sidorenko graph is common. Recently, Kr{\'a}l', Krnc and Lamaison proved that the $k$-apex of any tree is common for $k\leq 5$. In this paper, we extend their result by showing that for every positive integer $k$, the $k$-apex of any tree or any forest with average degree at least one is common. Our proof involves an application of the entropy method. 
\end{abstract}

\section{Introduction}\label{sec:introduction}

Ramsey’s theorem states that for any graph $H$, every $2$-edge-colouring of the complete graph $K_N$ contains a monochromatic copy of $H$, provided that $N$ is
sufficiently large. A natural quantitative question from Ramsey's theorem, which is called the \emph{Ramsey multiplicity problem}, is the following: What is the minimum number of monochromatic copies of $H$ contained in a $2$-edge-colouring of $K_N$?
For a fixed non-empty graph $H$, let $M(H,N)$ be the minimum number of labelled
monochromatic copies of $H$ over all $2$-edge-colourings of $K_N$.
The \emph{Ramsey
multiplicity constant} of $H$ is
\[
 C(H)=\lim_{N\to\infty}\frac{M(H,N)}{N^{v(H)}},
\]
where $v(H)$ denotes the number of vertices of
$H$. A uniformly random $2$-edge-colouring gives
$C(H)\le2^{1-e(H)}$, where $e(H)$ is the number of edges in $H$. The graph $H$ is called \emph{common} if equality
holds.

Determining which graphs are common is a difficult problem in extremal graph theory. Goodman~\cite{G59} proved that $K_3$ is common.
Erd\H{o}s~\cite{E62} conjectured that every complete graph is common,
and Burr and Rosta~\cite{BR80} extended this conjecture to all graphs.
Both conjectures were disproved: Sidorenko~\cite{S89} showed that a
triangle with a pendant edge is not common, and Thomason~\cite{T89}
showed that $K_4$ is not common. More generally, Jagger,
{\v S}{\v t}ov{\'\i}{\v c}ek and Thomason~\cite{JST96} proved that
every graph containing $K_4$ is not common.

On the positive side, for bipartite graphs, commonality is closely related to Sidorenko's conjecture~\cite{S93}, since every graph with the Sidorenko property is common. A bipartite graph $H$ has the Sidorenko property if the number of copies of $H$ in any graph is asymptotically at least
the number of its copies in the random graph with the same density. Sidorenko's conjecture states that every bipartite graph has the Sidorenko property. However, an OpenAI preprint~\cite{OAI26} recently claimed a counterexample to Sidorenko’s conjecture. For non-bipartite graphs, several families are known to be common. Hatami et al.~\cite{HHNRK12} proved that the $5$-wheel is common using the flag algebra method introduced by Razborov~\cite{R07}. Grzesik, Lee, Lidick{\'y} and Volec~\cite{GLLV} established
the commonality of several classes of tripartite graphs, including
triangle-trees. 
Kr{\'a}l', Volec and Wei~\cite{KVW} subsequently constructed connected common graphs with arbitrarily large chromatic number. 
For other interesting results on this topic; see, e.g.,~\cite{BMN22,BMN2024,Chen26,KKL26,KL22,KoL22,KNNVW22,LN23,R16}.

For a graph $H$ and a positive integer $k$, let $H^{+k}$ be the graph obtained by adding $k$ independent vertices, each adjacent to every vertex of $H$. We call $H^{+k}$ the \emph{$k$-apex} of $H$. Sidorenko~\cite{S96} showed that $H^{+1}$ is common whenever $H$ is a connected bipartite graph with the Sidorenko property.
Grzesik, Lee, Lidick{\'y} and Volec~\cite{GLLV} proposed the following extension.

\begin{conj}[Grzesik, Lee, Lidick{\'y} and Volec]
\label{conj:apices}
If $H$ is a connected bipartite graph with the Sidorenko property,
then $H^{+k}$ is common for every positive integer $k$.
\end{conj}

They proved this conjecture for connected bipartite graphs on at most
five vertices \cite[Theorem~1.4]{GLLV}. Recently, Kr{\'a}l', Krnc and
Lamaison~\cite{KKL26} proved that the $k$-apex of every tree is common
for $k\le5$. Our first result removes the restriction on $k$.

\begin{thm}\label{main}\label{thm:common}
For every tree $T$ and every positive integer $k$, the graph $T^{+k}$
is common.
\end{thm}

Our argument also applies to forests, including forests with isolated
vertices, provided that their average degree is at least one.
For a forest on $n$ vertices with $q$ components, this condition is
equivalent to $n\ge2q$.

\begin{thm}\label{thm:forest}
Let $F$ be a forest on $n$ vertices with $q\ge1$ components. If
$n\ge2q$, then $F^{+k}$ is common for every positive integer $k$.
\end{thm}

The proofs involve applications of the entropy method, which is a powerful tool in extremal combinatorics. In Section~\ref{sec:preliminaries}, we recall the entropy method as well as the theory of graph limits and Goodman's formula. In Section~\ref{sec:proofs}, we prove key counting inequalities by the entropy method and finally deduce our main results.

\section{Preliminaries}\label{sec:preliminaries}

We use the language of graph limits to formulate our results; see
\cite{LL12} for a comprehensive treatment of the theory of graph limits. All graphs are finite, simple,
and undirected. We write $v(H)$ and $e(H)$ for the numbers of vertices
and edges of a graph $H$, respectively, and $P_2$ for the path of length two. The triangular book $K_{1,1,k}$ consists of $k$ triangles
sharing an edge, called its spine.

For a graph $G$ with at least one vertex, a \emph{homomorphism} from
a graph $H$ to $G$ is a map
$\phi\colon V(H)\to V(G)$ that preserves adjacency. We denote the set
of such maps by $\Hom(H,G)$ and write $\hom(H,G)=|\Hom(H,G)|.$
Define the \emph{homomorphism density} of $H$ in $G$ as $$t(H,G)=\frac{\hom(H,G)}{v(G)^{v(H)}}.$$
Thus $t(H,G)$ is the probability that a uniformly random map from
$V(H)$ to $V(G)$ is a homomorphism. 
A sequence of graphs $(G_N)$ with
$v(G_N)\to\infty$ is \emph{convergent} if $t(H,G_N)$ converges for
every graph $H$. A \emph{graphon} is a measurable function $W\colon [0,1]^2 \to [0,1]$ such that $W(x,y)=W(y,x)$ for almost every $(x,y)\in [0,1]^2$. The homomorphism density of $H$ in $W$ is
\[
 t(H,W)=\int_{[0,1]^{V(H)}}
       \prod_{uv\in E(H)}W(x_u,x_v)\prod_{v\in V(H)}dx_v.
\]
For a finite graph $G$, let $W_G$ be the graphon obtained by partitioning
$[0,1]$ into $v(G)$ intervals of equal measure and setting $W_G$ equal
to the corresponding entry of the adjacency matrix of $G$ on each
rectangle. Then $t(H,W_G)=t(H,G)$ for every graph $H$. Every convergent graph sequence has a graphon limit: there is a graphon $W$ such that $t(H,G_N)$ tends to $t(H,W)$ for every $H$.

Conversely, every graphon $W$ can be obtained as the limit of a
sequence of finite graphs in the sense of homomorphism densities.
Such a sequence can be constructed using \emph{$W$-random graphs}:
choose $x_1,x_2,\ldots$ independently and uniformly from $[0,1]$
and, conditional on these points, join each pair $i<j$ independently
with probability $W(x_i,x_j)$. Let $G_N$ be the graph induced by
$\{1,\ldots,N\}$. It is standard that, with probability one,
for every finite graph $H$,
\begin{equation}\label{eq:w-random-convergence}
 t(H,G_N)\longrightarrow t(H,W),
 \qquad
 t(H,\overline{G_N})\longrightarrow t(H,1-W);
\end{equation}
see \cite[Section~2]{GLLV} and \cite[Proposition~11.32]{LL12}.
Consequently, any non-strict polynomial inequality involving
finitely many homomorphism densities in a graph and its complement
can be transferred from finite graphs to graphons by passing to
the limit along such a sequence. We use this observation to derive
graphon inequalities from discrete entropy arguments, as in
\cite[proof of Lemma~3.6]{GLLV}.

\subsection{Common graphs and Goodman's formula}
\label{subsec:common-graphs}

A red/blue edge-colouring of a complete graph is determined by its red
graph $G$, with $\overline G$ as its blue graph. In the graphon setting,
the two colours are represented by $W$ and $1-W$. Accordingly, define
the \emph{monochromatic homomorphism density} of $H$ by
\[
 m(H,W)=t(H,W)+t(H,1-W).
\]
For a colouring of $K_N$ with red graph $G$, the number of labelled
monochromatic copies of $H$ is
\[
 N^{v(H)}m(H,W_G)+O_H(N^{v(H)-1}),
\]
since at most $O_H(N^{v(H)-1})$ maps in either colour fail to be
injective. The error is uniform over all colourings. Together with \eqref{eq:w-random-convergence}, this gives the graphon formulation of the
Ramsey multiplicity constant:
\begin{equation}\label{eq:ramsey-graphon}
 C(H)=\min_W m(H,W),
\end{equation}
where the minimum is over all graphons. 
Indeed, compactness of the graphon space in the cut distance and continuity of $m(H,\cdot)$ ensure that the minimum is attained; see \cite{LL12}. Equality with
$C(H)$ follows by considering the $W$-random graphs described above.

The constant graphon $W\equiv1/2$ gives $m(H,W)=2^{1-e(H)}$.
Consequently, $H$ is common if and only if
\[
 m(H,W)\ge2^{1-e(H)}
\]
for every graphon $W$. In the language of graph limits, we say that a bipartite graph $H$ has the
\emph{Sidorenko property} if
$t(H,W)\ge t(K_2,W)^{e(H)}$ for every graphon $W$.

The following form of Goodman's formula~\cite{G59} relates the monochromatic
triangle density to the density of $P_2$.

\begin{lem}[Goodman's formula]
\label{lem:goodman}
For every graphon $W$, let $s=m(P_2,W)$. Then
\begin{equation}\label{eq:goodman}
 \frac12\le s\le1,
 \qquad m(K_3,W)=\frac{3s-1}{2}\ge s^2.
\end{equation}
In particular, $m(K_3,W)\ge1/4$.
\end{lem}
\begin{proof}
Let $d_W(x)=\int_0^1W(x,y)\,dy$. Then
\[
 s=\int_0^1\bigl(d_W(x)^2+(1-d_W(x))^2\bigr)\,dx,
\]
which gives $1/2\le s\le1$. Expanding $t(K_3,1-W)$ yields
\[
 m(K_3,W)=1-3t(K_2,W)+3t(P_2,W)=\frac{3s-1}{2}.
\]
Finally, $(3s-1)/2-s^2=(2s-1)(1-s)/2\ge0$.
\end{proof}

\subsection{The entropy method}\label{subsec:entropy}

We use the entropy method to prove homomorphism counting inequalities. For applications of this method in inequalities of homomorphism densities, see, e.g.,~\cite{BCNS25,BMN2024,CCN,CCN26,CKLL18,GLLV,Lee21,Sz15}.
We recall the definitions and standard properties of this method, which can also be found in~\cite[Section 15]{AS16}. Readers familiar with the basic properties of entropy may skip directly to Lemmas~\ref{lem:distribution-gluing} and~\ref{lem:forest-sampling}. All random variables used in the entropy arguments are discrete and have finite range.

For a discrete random variable $X$, its \emph{range} is
\[
 \rng(X)=\{x:\Prob(X=x)>0\}.
\]
The \emph{entropy} of $X$ is
\[
 \ent(X)=-\sum_{x\in\rng(X)}\Prob(X=x)\log_2\Prob(X=x).
\]
For discrete random variables $X,Y$ and $y\in\rng(Y)$, define the
\emph{range of $X$ given that $Y=y$} by
\[
 \rng(X\mid Y=y)=\{x:\Prob(X=x\mid Y=y)>0\}.
\]
The \emph{conditional entropy of $X$ given that $Y=y$} is
\[
 \ent(X\mid Y=y)
 =-\sum_{x\in\rng(X\mid Y=y)}
     \Prob(X=x\mid Y=y)\log_2\Prob(X=x\mid Y=y).
\]
The \emph{conditional entropy of $X$ given $Y$} is
\[
 \ent(X\mid Y)=\sum_{y\in\rng(Y)}\Prob(Y=y)\ent(X\mid Y=y).
\]

We say that discrete random variables $X$ and $Z$ are
\emph{conditionally independent given $Y$} if
\[
 \begin{aligned}
 \Prob(X=x\mid Y=y,Z=z)&=\Prob(X=x\mid Y=y),\\
 \Prob(Z=z\mid Y=y,X=x)&=\Prob(Z=z\mid Y=y)
 \end{aligned}
\]
for all $y\in\rng(Y)$, $x\in\rng(X\mid Y=y)$ and
$z\in\rng(Z\mid Y=y)$. We use the following standard entropy
properties.

\begin{lem}\label{lem:entropy-facts}
The following hold for discrete random variables with finite range.
\begin{enumerate}
\item \emph{(Maximality of the Uniform Distribution)}
For every discrete random variable $X$,
\[
 \ent(X)\le\log_2|\rng(X)|,
\]
with equality if and only if $X$ is uniformly distributed on
$\rng(X)$.
\item \emph{(Chain Rule)}
For any discrete random variables $X_1,\ldots,X_r$,
\[
 \ent(X_1,\ldots,X_r)
 =\ent(X_1)+\sum_{i=2}^r\ent(X_i\mid X_1,\ldots,X_{i-1}).
\]
\item \emph{(Deconditioning Lemma)}
For any discrete random variables $X,Y,Z$,
\[
 \ent(X\mid Y,Z)\le\ent(X\mid Y),
\]
with equality if and only if $X$ and $Z$ are conditionally independent
given $Y$.
\end{enumerate}
\end{lem}

The Chain Rule and the Deconditioning Lemma imply
\begin{equation}\label{eq:conditional-subadditivity}
 \ent(Z_1,\ldots,Z_k\mid X)
 \le\sum_{i=1}^k\ent(Z_i\mid X),
\end{equation}
with equality when $Z_1,\ldots,Z_k$ are conditionally independent
given $X$.

A random homomorphism $\Phi\colon H\to G$ and the tuple
$(\Phi(v))_{v\in V(H)}$ encode the same information and therefore
have the same entropy. Since $\rng(\Phi)\subseteq\Hom(H,G)$,
Maximality of the Uniform Distribution gives
\begin{equation}\label{eq:entropy-counting}
 \ent(\Phi)\le\log_2|\rng(\Phi)|\le\log_2\hom(H,G).
\end{equation}
Equality in the bound $\ent(\Phi)\le\log_2\hom(H,G)$ holds if and
only if $\Phi$ is uniformly distributed on $\Hom(H,G)$. Thus a lower
bound on the entropy of a random homomorphism gives a lower bound
on the number of homomorphisms.

We also need to construct new tuples of random variables from given
ones while preserving their joint distributions on common
coordinates. The following standard lemma allows us to do this.

\begin{lem}[Distribution Gluing Lemma;
 {see, e.g., \cite[Lemma~2.5]{Lee21} and~\cite[Lemma~3.12]{CCN}}]\label{lem:distribution-gluing}
Let $(A_0,B_1,B_2,C)$ be a tuple of discrete random variables such that
$B_1$ and $B_2$ are identically distributed. Then there exists a
random variable $C'$ such that $(B_1,C')$ and $(B_2,C)$ are
identically distributed and $A_0$ and $C'$ are conditionally
independent given $B_1$. Moreover,
\begin{equation}\label{eq:gluing-entropy}
 \ent(A_0,B_1,C')=\ent(A_0,B_1)+\ent(B_2,C)-\ent(B_2).
\end{equation}
\end{lem}
\begin{proof}
First sample $(A_0,B_1)$ according to its joint distribution. Given
$A_0=a$ and $B_1=b$, sample $C'$ according to the distribution of $C$
conditioned on $B_2=b$. In other words,
\[
 \Prob(C'=c\mid A_0=a,B_1=b)=\Prob(C=c\mid B_2=b).
\]
This is well-defined for every $b\in\rng(B_1)$ because $B_1$ and
$B_2$ are identically distributed. For all such $b$ and all $c$,
\[
 \Prob(B_1=b,C'=c)
 =\Prob(B_1=b)\Prob(C=c\mid B_2=b)
 =\Prob(B_2=b,C=c).
\]
Hence $(B_1,C')$ and $(B_2,C)$ are identically distributed. The
sampling rule also makes $A_0$ and $C'$ conditionally independent
given $B_1$. By the Chain Rule,
\[
 \ent(A_0,B_1,C')=\ent(A_0,B_1)+\ent(C'\mid B_1)
 =\ent(A_0,B_1)+\ent(C\mid B_2),
\]
which gives \eqref{eq:gluing-entropy}.
\end{proof}

Repeated applications of the Distribution Gluing Lemma give the
following construction for forests. For a tree, it is a special
case of the Markov-tree construction in~\cite[Theorem~2.6]{Lee21} and~\cite[Lemma~3.5]{GLLV}.

\begin{lem}[Gluing along a forest]\label{lem:forest-sampling}
Let $(X,Y,A)$ be a tuple of discrete random variables whose joint
distribution is invariant under interchanging $X$ and $Y$. For every
forest $F$ on $n$ vertices with $q\ge1$ components, there is a random
tuple $\boldsymbol{U}=(A,(X_v)_{v\in V(F)})$ such that $(X_v,A)$ has
the distribution of $(X,A)$ for every vertex $v$, and $(X_u,X_v,A)$
has the distribution of $(X,Y,A)$ for every oriented edge $uv$.
Moreover,
\begin{equation}\label{eq:forest-entropy}
 \ent(\boldsymbol{U})
 =(n-q)\ent(X,Y,A)-(n-2q)\ent(X,A)-(q-1)\ent(A).
\end{equation}
\end{lem}
\begin{proof}
Choose a root in each component and relabel the vertices so that
$v_1,\ldots,v_q$ are the roots and every parent precedes its
children. Starting with $\boldsymbol{U}_0=(A)$, we construct
\[
 \boldsymbol{U}_i=(A,X_{v_1},\ldots,X_{v_i})
 \qquad (1\le i\le n)
\]
recursively. At step $i$, apply Lemma~\ref{lem:distribution-gluing}
with $\boldsymbol{U}_{i-1}$ as the variable denoted by $A_0$ in that lemma.

For $i\le q$, take $B_1=B_2=A$ and $C=X$. The lemma supplies a
variable $X_{v_i}$ such that $(X_{v_i},A)$ and $(X,A)$ are
identically distributed, with entropy increment
\[
 \ent(\boldsymbol{U}_i)-\ent(\boldsymbol{U}_{i-1})
 =\ent(X,A)-\ent(A)=\ent(X\mid A).
\]
For $i>q$, let $v_j$ be the parent of $v_i$, where $j<i$.
Inductively, $(X_{v_j},A)$ and $(X,A)$ are identically distributed.
Apply the same lemma with $B_1=(X_{v_j},A)$, $B_2=(X,A)$ and $C=Y$
to append $X_{v_i}$ to $\boldsymbol{U}_{i-1}$. Then
$(X_{v_j},X_{v_i},A)$ and $(X,Y,A)$ are identically distributed, and
\[
 \ent(\boldsymbol{U}_i)-\ent(\boldsymbol{U}_{i-1})
 =\ent(X,Y,A)-\ent(X,A)=\ent(Y\mid X,A).
\]
By symmetry, $(X_{v_i},A)$ also has the distribution of $(X,A)$,
which continues the induction. Each step preserves the joint
distribution of the variables already constructed, so
$\boldsymbol{U}=\boldsymbol{U}_n$ has all the required marginals;
symmetry gives both orientations of every edge.

There are $q$ root steps and $n-q$ nonroot steps. Summing their
entropy increments gives
\begin{align*}
 \ent(\boldsymbol{U})
 &=\ent(A)+q\ent(X\mid A)+(n-q)\ent(Y\mid X,A)\\
 &=(n-q)\ent(X,Y,A)-(n-2q)\ent(X,A)-(q-1)\ent(A),
\end{align*}
as required.
\end{proof}
\section{Proofs of the main results}\label{sec:proofs}

\subsection{Proof ideas}
We first recall the inequalities used in the proof by Kr{\'a}l', Krnc and Lamaison~\cite{KKL26} that $T^{+k}$ is common for $1\le k \le 5$. Then we give an entropy interpretation of their argument for trees and introduce our new ideas. 

Let $T$ be a tree on $n\ge 2$ vertices, and write
$B_k=K_{1,1,k}$ and $S_k=K_{1,k}$.
Kr{\'a}l', Krnc and Lamaison~\cite{KKL26} apply Sidorenko's
inequality inside monochromatic common neighbourhoods and then
H\"older's inequality to obtain
\[
m(T^{+k},W)\ge
\frac{m(B_k,W)^{n-1}}{m(S_k,W)^{n-2}}.
\]
For trees, this estimate has the following entropy interpretation.
In a finite graph $G$ containing a triangle, choose
$(\widehat X,\widehat Y,\widehat A)$ uniformly from
$\Hom(B_k,G)$, where $\widehat X,\widehat Y$ represent the spine
and $\widehat A$ is the ordered tuple of apex images.
Root $T$ at $r$ and first sample
$(\widehat X_r,\widehat A)$ according to the induced star
distribution. For each parent $u$ and child $v$, sample
$\widehat X_v$ according to the conditional distribution of
$\widehat Y$ given
$(\widehat X,\widehat A)=(\widehat X_u,\widehat A)$,
independently of the previously sampled variables conditional
on this pair. The apex tuple is shared throughout the construction.
The resulting random homomorphism $\widehat\Phi$ satisfies
\[
\ent(\widehat\Phi)
=(n-1)\ent(\widehat X,\widehat Y,\widehat A)
 -(n-2)\ent(\widehat X,\widehat A).
\]
The book entropy equals $\log_2\hom(B_k,G)$, while Maximality
of the Uniform Distribution gives
$\ent(\widehat X,\widehat A)\le\log_2\hom(S_k,G)$.
These bounds yield the finite-graph counterpart of the estimate above.

To deduce commonality, Kr{\'a}l', Krnc and Lamaison combine
this estimate with commonality of $B_k$ and the inequality
\[
\frac{m(B_k,W)}{m(S_k,W)}\ge 2^{-(k+1)}.
\]
They prove this ratio bound for $k\le5$, but it fails for $k=12$; see \cite[Section~4]{KKL26}. Thus this ratio estimate cannot extend their argument to all $k$. 
In the entropy interpretation, the gluing construction retains the induced star marginal, but
bounding its entropy by $\log_2\hom(S_k,G)$ uses only its support and discards its actual probabilities.

We retain the same gluing construction and change the local
book distribution. The key step is choosing $(X,Y,Z)$ uniformly from
$\Hom(K_3,G)$ and, conditional on $(X,Y)$, taking independent
copies $Z_1,\ldots,Z_k$ of $Z$. Put $A=(Z_1,\ldots,Z_k)$.
The resulting book is generally nonuniform, but every page
$(X,Y,Z_i)$ is a uniform triangle. In particular,
$(X,Z_i)$ and $(X,Y)$ are identically distributed.
Consequently, the Chain Rule and the Deconditioning Lemma give
\[
\begin{aligned}
\ent(X,Y,A)
 &=k\ent(X,Y,Z)-(k-1)\ent(X,Y),\\
\ent(X,A)
 &\le k\ent(X,Y)-(k-1)\ent(X).
\end{aligned}
\]
Thus the star entropy is controlled by the same triangle-projected edge distribution that appears in the book entropy. After gluing, the remaining entropy terms can be bounded by edge and two-edge path counts. Combining the two colours, Goodman's inequality provides the triangle-density bound needed to absorb the path factor and recover the commonality threshold.

\subsection{Proofs of Theorem~\ref{thm:common} and Theorem~\ref{thm:forest}}
We first establish a key inequality for apex extensions of
forests. Combining the two colours then gives a lower bound in terms
of $m(P_2,W)$, from which both main theorems follow. Throughout this section, factors with exponent zero are omitted.

\begin{lem}[The key inequality]\label{lem:one-colour}
Let $F$ be a forest on $n$ vertices with $q\ge1$ components, where
$n\ge2q$, and let $k\ge1$. Set
\[
 a=k(n-q),\qquad b=(k-1)(n-2q),\qquad c=n+q(k-3).
\]
For every graphon $W$,
\begin{equation}\label{eq:forest-one-colour}
 t(F^{+k},W)t(P_2,W)^b t(K_2,W)^c\ge t(K_3,W)^a.
\end{equation}
In particular, if $T$ is a tree on $n\ge2$ vertices, then
\begin{equation}\label{eq:one-colour}
 t(T^{+k},W)t(P_2,W)^{(k-1)(n-2)}t(K_2,W)^{k+n-3}
 \ge t(K_3,W)^{k(n-1)}.
\end{equation}
\end{lem}
\begin{proof}
The exponents satisfy
\[
 a\ge q\ge1,\qquad b\ge0,\qquad
 c=(n-2q)+q(k-1)\ge0.
\]
By \eqref{eq:w-random-convergence}, it suffices to prove
\eqref{eq:forest-one-colour} for finite graphs: after normalization,
we may apply the inequality to a sequence of $W$-random graphs and
pass to the limit, as in \cite[proof of Lemma~3.6]{GLLV}.

Let $G$ be a finite graph with $N\ge1$ vertices. If $G$ contains no
triangle, the right-hand side is zero and there is nothing to prove.
Otherwise, let $(X,Y,Z)$ be a uniformly random element of
$\Hom(K_3,G)$. Its distribution is invariant under permutations of
the coordinates. By the definition of $(X,Y,Z)$ and the Maximality of the Uniform Distribution,
\begin{equation}\label{eq:triangle-edge-entropy}
 \ent(X,Y,Z)=\log_2\hom(K_3,G),
 \qquad \ent(X,Y)\le\log_2\hom(K_2,G).
\end{equation}

Given a realization $(X,Y)=(x,y)$, choose $Z_1,\ldots,Z_k$
independently and uniformly, with replacement, from the common
neighbourhood $N_G(x)\cap N_G(y)$. This set is nonempty whenever $\Prob(X=x,Y=y)>0$. Since $(X,Y,Z)$ is a uniformly random homomorphism from $K_3$ to $G$, the uniform distribution on this common neighbourhood is precisely the conditional distribution
of $Z$ given $(X,Y)=(x,y)$. Then put
$A=(Z_1,\ldots,Z_k)$. Thus $(X,Y,A)$ is a random homomorphism from $K_{1,1,k}$ to $G$, symmetric in $X,Y$. The Chain Rule and conditional independence give
\begin{equation}\label{eq:book-entropy}
 \ent(X,Y,A)
 =\ent(X,Y)+k\ent(Z\mid X,Y)
 =k\ent(X,Y,Z)-(k-1)\ent(X,Y).
\end{equation}
By construction, $(X,Y,Z_i)$ has the same distribution as
$(X,Y,Z)$ for every $i$. Since the uniform distribution on
$\Hom(K_3,G)$ is invariant under interchanging the last two
coordinates, $(X,Z)$ and $(X,Y)$ are identically distributed.
Consequently, $(X,Z_i)$ has the distribution of $(X,Y)$.
The Chain Rule and the Deconditioning Lemma give
\begin{equation}\label{eq:star-entropy}
\begin{aligned}
 \ent(X,A) &\le \ent(X)+\sum_{i=1}^k \ent(Z_i|X)=\ent(X)+k\ent(Y\mid X)\\
 &=k\ent(X,Y)-(k-1)\ent(X).
 \end{aligned}
\end{equation}

To bound the remaining entropy expression, let $Y'$ be a
conditionally independent copy of $Y$ given $X$. The tuple
$(Y,X,Y')$ is a random homomorphism from $P_2$ to $G$, so
\begin{equation}\label{eq:path-entropy}
 2\ent(X,Y)-\ent(X)
 =\ent(X)+2\ent(Y\mid X)
 =\ent(Y,X,Y')\le\log_2\hom(P_2,G).
\end{equation}

Apply Lemma~\ref{lem:forest-sampling} to $(X,Y,A)$ and $F$, and let
$\boldsymbol{U}=(A,(X_v)_{v\in V(F)})$ be the resulting tuple.
For each edge $uv$ of $F$, the tuple $(X_u,X_v,A)$ has the book
distribution. For every vertex $v$, including an isolated vertex,
$(X_v,A)$ has the distribution of $(X,A)$; hence $X_v$ is adjacent
to every entry of $A$. Thus $\boldsymbol{U}$ is a random
homomorphism from $F^{+k}$ to $G$.

Using \eqref{eq:forest-entropy}, \eqref{eq:book-entropy} and
\eqref{eq:star-entropy}, with $n-2q\ge0$, gives
\begin{align*}
 &\log_2\hom(F^{+k},G)
 \ge\ent(\boldsymbol{U})\\
 &=(n-q)\ent(X,Y,A)-(n-2q)\ent(X,A)-(q-1)\ent(A)\\
 &\ge a\ent(X,Y,Z)-c\ent(X,Y)-b\bigl(2\ent(X,Y)-\ent(X)\bigr)-(q-1)\ent(A)\\
 &\ge a\log_2\hom(K_3,G)-c\log_2\hom(K_2,G)-b\log_2\hom(P_2,G)-k(q-1)\log_2 N.
\end{align*}
The last step follows from \eqref{eq:triangle-edge-entropy},
\eqref{eq:path-entropy} and $\ent(A)\le k\log_2 N$. Exponentiating,
we obtain
\[
 \hom(F^{+k},G)\hom(P_2,G)^b\hom(K_2,G)^c N^{k(q-1)}
 \ge\hom(K_3,G)^a.
\]
Since
\[
 n+k+3b+2c+k(q-1)=3a,
\]
division by $N^{3a}$ proves \eqref{eq:forest-one-colour} for $G$.
Applying this inequality to $W$-random graphs and using
\eqref{eq:w-random-convergence} proves it for $W$. Taking $q=1$ gives
\eqref{eq:one-colour}.
\end{proof}

The following lemma, combining Lemma~\ref{lem:one-colour}, Goodman's formula and H\"older's inequality, gives lower bounds for $m(T^{+k},W)$ and $m(F^{+k},W)$ in terms of $m(P_2,W)$.

\begin{lem}\label{lem:two-colour}
Let $F$ be a forest on $n$ vertices with $q\ge1$ components, where
$n\ge2q$. For every positive integer $k$ and every graphon $W$,
\begin{equation}\label{eq:forest-degree}
 m(F^{+k},W)\ge2^{1-q}m(P_2,W)^{e(F^{+k})-q}.
\end{equation}
Moreover, for every tree $T$, every positive integer $k$ and every
graphon $W$,
\begin{equation}\label{eq:tree-degree}
 m(T^{+k},W)\ge m(P_2,W)^{e(T^{+k})-1}.
\end{equation}
\end{lem}
\begin{proof}
Let $W_0=W$ and $W_1=1-W$. Let $a,b,c$ be as in Lemma~\ref{lem:one-colour}. Since $a\geq q \geq 1$, $b,c\geq 0$ and $q+b+c=a$, the weighted form of H\"older's inequality gives
\[
 \sum_{i=0}^1 x_i^{q/a}y_i^{b/a}z_i^{c/a}
 \le
 \left(\sum_{i=0}^1x_i\right)^{q/a}
 \left(\sum_{i=0}^1y_i\right)^{b/a}
 \left(\sum_{i=0}^1z_i\right)^{c/a}
\]
for nonnegative $x_i,y_i,z_i$.
Applying this with
\[
 x_i=t(F^{+k},W_i)^{1/q},\qquad
 y_i=t(P_2,W_i),\qquad
 z_i=t(K_2,W_i)
\]
and using Lemma~\ref{lem:one-colour}, we obtain
\begin{align*}
 m(K_3,W)
 &\le\sum_{i=0}^1 t(F^{+k},W_i)^{1/a}
          t(P_2,W_i)^{b/a}t(K_2,W_i)^{c/a}\\
 &\le\left(\sum_{i=0}^1t(F^{+k},W_i)^{1/q}\right)^{q/a}
          m(P_2,W)^{b/a}m(K_2,W)^{c/a}\\
 &\le2^{(q-1)/a}m(F^{+k},W)^{1/a}m(P_2,W)^{b/a}.
\end{align*}
The last step uses concavity of $x^{1/q}$
and $m(K_2,W)=1$. Write $s=m(P_2,W)$. By
Lemma~\ref{lem:goodman}, $s\ge1/2$ and $m(K_3,W)\ge s^2$, so
\begin{equation}\label{eq:forest-two-colour}
 m(F^{+k},W)\ge2^{1-q}\frac{m(K_3,W)^a}{s^b}
 \ge2^{1-q}s^{2a-b}.
\end{equation}
Since $e(F^{+k})=(k+1)n-q$ and
$2a-b=(k+1)n-2q=e(F^{+k})-q$, this proves
\eqref{eq:forest-degree}.

If $T$ has at least two vertices, \eqref{eq:tree-degree} follows by
setting $q=1$. It remains to consider the one-vertex tree, for which
$T^{+k}=K_{1,k}$. The cases $k=1,2$ are immediate.
For $k\ge3$, write $d(x)=\int_0^1W(x,y)\,dy$.
Since $\int_0^1[d(x)+(1-d(x))]\,dx=1$, H\"older's inequality gives
\begin{align*}
 m(K_{1,k},W)
 &=\int_0^1\bigl(d(x)^k+(1-d(x))^k\bigr)\,dx\\
 &\ge\left(\int_0^1\bigl(d(x)^2+(1-d(x))^2\bigr)\,dx\right)^{k-1}\\
 &=m(K_{1,2},W)^{k-1}.
\end{align*}
This proves \eqref{eq:tree-degree} for the one-vertex tree.
\end{proof}

\begin{proof}[Proof of Theorems~\ref{main} and~\ref{thm:forest}]
For every graphon $W$, Lemmas~\ref{lem:goodman} and~\ref{lem:two-colour}
give, for a tree $T$,
\[
 m(T^{+k},W)\ge m(P_2,W)^{e(T^{+k})-1}
 \ge2^{1-e(T^{+k})}.
\]
For a forest $F$ satisfying the hypotheses of Theorem~\ref{thm:forest},
the same lemmas give
\[
 m(F^{+k},W)\ge2^{1-q}m(P_2,W)^{e(F^{+k})-q}
 \ge2^{1-q}2^{-(e(F^{+k})-q)}
 =2^{1-e(F^{+k})}.
\]
The exponents are nonnegative under the respective hypotheses.
Hence $T^{+k}$ and $F^{+k}$ are common for every integer $k \geq 1$.
\end{proof}

\begin{rem}\label{rem:forest-restriction}

The condition $n\ge2q$ in Theorem~\ref{thm:forest} is equivalent to
average degree at least one. In the proof, it ensures that the upper
bound on $\ent(X,A)$ in \eqref{eq:star-entropy} gives a lower bound
on the total entropy in \eqref{eq:forest-entropy}. Some restriction
on the forest is necessary: the one-apex of $K_1\sqcup K_2$ is a
triangle with a pendant edge, which is not common.
\end{rem}

\begin{rem}\label{rem:sidorenko-bases}
Kr{\'a}l', Krnc and Lamaison~\cite{KKL26} in fact prove that
$H^{+k}$ is common for every graph $H$ with the Sidorenko property and $e(H)=v(H)-1$, provided that $1\le k\le5$.
Their result does not require $H$ to be connected, so the base graph may contain cycles. Indeed, they apply Sidorenko's inequality to $H$ inside each monochromatic common neighbourhood of the apex tuple. 
The subsequent H\"older argument depends on $H$ only through
$v(H)$ and $e(H)$; the relation $e(H)=v(H)-1$ therefore gives the same estimate as for trees. In contrast, our forest-gluing
construction relies on acyclicity and does not cover the cyclic
bases allowed by their theorem.
\end{rem}

\section*{AI Declaration}
The entropy methods used in this paper are inspired by Clemen, Noel and the author's earlier works~\cite{CCN} and~\cite{CCN26}, as well as Grzesik, Lee, Lidick{\'y} and Volec's paper~\cite{GLLV}. At first, the author only aimed to prove the commonality of the $k$-apex of trees. The original proof of Theorem~\ref{thm:common} was developed by the author. After proving Theorem~\ref{thm:common}, the author input the proof into ChatGPT-6.1 and asked for further extensions. ChatGPT-6.1 then extended Theorem~\ref{thm:common} to all forests with average degree at least one. The idea and argument of the forest proof constitute a simple generalization of those in the proof of Theorem~\ref{thm:common} (mainly see Lemma~\ref{lem:forest-sampling}); therefore, the author ultimately decided to present them as a unified proof. In addition, the author used AI to review the entire paper, proofread the text for errors and omissions, and polish the language. The author takes full responsibility for the correctness of the paper.

\end{document}